\documentclass[a4paper,11pt,english]{article}

\usepackage{fourier}
\usepackage[margin=2.5cm,centering,vcentering]{geometry}
\usepackage[english]{babel}
\usepackage[utf8]{inputenc}
\usepackage[T1]{fontenc} 

\usepackage{microtype}
\usepackage{amsfonts,amssymb,amsmath,amsthm}
\usepackage{stmaryrd} 
\usepackage[svgnames]{xcolor}
\usepackage{mathrsfs}
\usepackage{mdframed}
\usepackage{subcaption}
\usepackage{graphicx}
\usepackage{float}
\usepackage{framed}
\usepackage{tikz}
\usetikzlibrary{arrows.meta}
\usepackage{pgfplots}

\usepackage{fancyhdr}

\usepackage{hyperref}
\hypersetup{
    colorlinks=true, 
    linktoc=all,     
    colorlinks,
    citecolor=black,
    filecolor=black,
    linkcolor=black,
    urlcolor=black
}

\theoremstyle{definition}
\newtheorem{theoreme}{Theorem}

\newtheorem*{ex}{Example}

\newcommand{\N} {\mathbb{N}}

\newcommand{\F} {\mathbb{F}}
\newcommand{\G} {\mathbb{G}}

\newcommand{\att}{K\!\left[\!\left[\frac{1}{t}\right]\!\right]}

\newcommand{\btt}{K\!\left(\!\left(\frac{1}{t}\right)\!\right)}

\pgfplotsset{
	compat=1.18,
	every axis/.append style={
		height=\pgfkeysvalueof{/pgfplots/width}*(((\pgfkeysvalueof{/pgfplots/ymax})+((-1)*(\pgfkeysvalueof{/pgfplots/ymin})))/((\pgfkeysvalueof{/pgfplots/xmax})+((-1)*(\pgfkeysvalueof{/pgfplots/xmin}))))+0.1pt, 
		axis x line=bottom,
		axis y line = left,
		axis lines=middle,							
		axis line style={-{Stealth}},				
		scale only axis,							 
		enlargelimits={0.1},
		grid style={densely dashed,line width=0.4pt,draw=black!30},
		minor grid style={line width=0.25pt,draw=black!20},
		axis equal,
		legend pos=north east,
	},
		every axis plot post/.append style={
		semithick,
	},
}

\pgfkeys{/pgf/number format/use comma,/pgf/number format/.cd,1000 sep={\,}}

\author{LUCAS Alexis}

\begin{document}

    {\huge\bfseries Purity and almost strict purity of Anderson $t$-modules \\[0.4cm]}
\textbf{Alexis Lucas*}

\begin{center}
\begin{minipage}{0.9\linewidth}
\begin{small}
\itshape
Normandie Univ, UNICAEN, CNRS, LMNO, 14000 Caen, France \\
E-mail: alexis.lucas@unicaen.fr \\[5mm]
\textbf{Abstract.} We study the relations between the notion of purity of a $t$-module introduced by Anderson and that of almost strict purity for a $t$-module introduced by Namoijam and Papanikolas.
\end{small}
\end{minipage}
\end{center}

\section{Introduction}
In \cite{Anderson}, G.W. Anderson defined $t$-modules and the notion of purity. Related to $t$-modules, C. Namoijam and M. A. Papanikolas defined the notion of almost strict purity in \cite[Remark 4.5]{derivative}, then they proved that an almost strictly pure $t$-module is pure. We are interested here in the reciprocal, and with the help of the work of  A. Maurischat in \cite{maurischat} we show that these two notions are not equivalent by presenting a counter-example (see Theorem \ref{4}).

\section{Purity}
Let $K$ a perfect field containing $\F_q$. We let $\tau:K\rightarrow K$ denote the $q$-th  power Frobenius map and $K\{\tau\}$ be the ring of twisted polynomials in $\tau$ over $K$, subject to the relation, $\tau a =a^q\tau$ for any $a\in K$. Consider $\ell:\F_q[t]\rightarrow K$ a homomorphism of $\F_q$-algebras and denote $\sigma=\tau^{-1}$.

A $t$-module $(E, \varphi)$ over $K$ of dimension $d$ is by definition an $\F_q$-vector space scheme $E$ over
$K$ isomorphic to $ \G_a^d$ together with a homomorphism of $\F_q$-algebras $\varphi : \F_q[t] \rightarrow \operatorname{End}_{\operatorname{grp},\F_q}(E)$ into the ring of $\F_q$-vector space scheme endomorphisms of $E$, such that for all $a \in \F_q[t]$, the endomorphism $d\varphi_a$ on $\operatorname{Lie}(E)$ induced by $\varphi_a$ fulfills the condition that $d\varphi_a - \ell(a)$ is nilpotent.

We will fix in the following $(E,\phi)$ a $t$-module on $K$ of dimension $d$ as well as a coordinate system $\kappa$, i.e. an isomorphism of schemes in $\F_q$-vector spaces $\kappa:E\simeq \G_a^d$ defined on $K$. With respect to this coordinate system, we can represent $\phi_t$ by a matrix $D\in M_d(K\{\tau\}).$

Let $\operatorname{pr}_i:\G_a^d\rightarrow \G_a$ ($1\leq i\leq d)$ be the projection to the $i$-th component of $\G_a^d$, and let $\kappa_i=\kappa\circ \operatorname{pr}_i$. Let $\check{\kappa}_j:\G_a\rightarrow E$ be defined by $\check{\kappa}_j=\kappa^{-1}\circ \operatorname{inj}_j$ where $\operatorname{inj}_j:\G_a\rightarrow \G_a^d$ is the natural injection into the $j-$th component.

We say that $E$ is almost strictly pure if there is some integer $s\geq1$ such that 
$$D^s=A_0+A_1\tau+...+A_r\tau^r$$ with $A_r\in GL_d(K)$.

 The $t$-motive $M(E)$ of $E$ is the free $K\{\tau\}$-module of rank $d$ with base $\{\kappa_1,...,\kappa_d\}$ with a $t$-action on this base defined by
$$t.\begin{pmatrix} \kappa_1 \\ \vdots \\ \kappa_d\end{pmatrix}=D\begin{pmatrix} \kappa_1 \\ \vdots \\ \kappa_d\end{pmatrix}.$$
We define in a similar way the dual $t$-motive $\mathscr{M}$ as the free $K\{\tau\}$-module of rank $d$ of basis $\{\check{\kappa}_1,...,\check{\kappa}_d\}$ whose $t$-action  (on the right) on this basis is defined by
$$\begin{pmatrix} \check{\kappa}_1 & \cdots & \check{\kappa}_d\end{pmatrix}.t=\begin{pmatrix} \check{\kappa}_1 & \cdots & \check{\kappa}_d\end{pmatrix}D.$$
We say that $E$ is abelian if $M(E)$ is a finitely generated $K[t]$-module. In this case, we define $w(M)$ the weight of $M$ by
$$w(M)=\frac{d}{\operatorname{rk}(E)}$$ where $\operatorname{rk}(E)$ is the rank of $M$ as a $K[t]$-module (that is finite because $E$ is abelian). \\
We moreover consider:
\begin{itemize}
\item The ring of formal power series in $\frac{1}{t}$ with coefficients in $K$ denoted by $\att$.
\item The field of Laurent series in $\frac{1}{t}$ with coefficients in $K$ denoted by $\btt$ (that is the field of fractions of $\att$).
\end{itemize}

 The $t$-motive $M(E)$ and the $t$-module $E$ are called pure if there exists a $\att$-lattice $\Lambda$ in $\btt \otimes_{K[t]}M$ as well as positive integers $u,v\in \N$ such that
$$t^u\Lambda=\tau^v\Lambda.$$

We will use the following result, proved by A. Maurischat  in \cite[Theorem 6.6, Theorem 7.2]{maurischat}, characterizing the fact of being abelian and being pure using Newton polygons.

\begin{theoreme}\label{3}
    The $t$-module $E$ is abelian if and only if  the Newton polygon $N_{\lambda_d}$ of the last invariant factor $\lambda_d$ of the matrix $D$ has only positive slopes. In this case, $E$ is pure if and only if $N_{\lambda_d}$ has exactly one edge. Then we have that the weight of $M$ equals the reciprocal of the slope of the edge.
\end{theoreme}
Here we recall that the invariant factors of a matrix $D\in M_d(K\{\tau\})$ are obtained by diagonalizing the matrix $tI_d-D\in M_d(K(\{\sigma\})[t])$ by performing elementary operations on the rows and columns in $K(\{\sigma\})[t]$. Contrary to the commutative case the invariant factors are only unique up to similarity. 
\begin{ex}
In \cite{maurischat}, Maurischat defined the $t$-module given by the matrix
$$M:=\begin{pmatrix}
 \theta & 0 \\
 1& \theta \end{pmatrix}+\begin{pmatrix} 0 & 0 \\1 &0 \end{pmatrix}.\tau+\begin{pmatrix}
 1 & 0 \\
0 & 1 \end{pmatrix}.\tau^2+\begin{pmatrix}
0 & 1\\
 0 & 0 \end{pmatrix}.\tau^3\in M_2(K\{\tau\}).$$
By diagonalizing the matrix $tI_2-M$ we get the matrix
 $$\begin{pmatrix}1 &0 \\
 0 & \lambda_2\end{pmatrix}$$
 where
$$\lambda_2= \left(-\sigma^{-3}+(\theta+\theta^{q^2})\sigma^{-2}+\theta^{q^3+1}\right)-\left(2\sigma^{-2}+\theta^{q^{-3}}+\theta\right).t+t^2\in K(\{\sigma\}).$$
 If $\operatorname{char}(K)=2$ then we represent the Newton polygon of $\lambda_2$ in Figure 1. It has only one edge of slope $\dfrac{3}{2}$ hence by Theorem \ref{3} the $t$-module is abelian and pure of weight $\dfrac{2}{3}.$

\end{ex}

The authors of \cite{derivative} showed in the same paper the next result.
\begin{theoreme}\label{pspur} With the previous notation, an almost strictly pure $t$-module  is pure of weight $\frac{s}{r}$.
\end{theoreme}
We now turn our interest to the reciprocal of the above result, and we answer negatively.

\begin{theoreme}\label{4} For any integer $d\geq 2$ there exists a pure but not almost strictly pure $t$-module of dimension $d$.
\end{theoreme}

\begin{proof}
We first consider the case $d=2$. Let us note $\theta=\ell(t)$. Consider the $t$-module given by the matrix 
 $$D_2=\begin{pmatrix} 1 & 0 \\ 0&1\end{pmatrix}+\begin{pmatrix}1 & 0 \\ 0& 1\end{pmatrix}.\tau+\begin{pmatrix} 0 & 0 \\ \theta & 0 \end{pmatrix}.\tau^2\in M_d(K\{\tau\}).$$
Let us diagonalize the matrix  $tI_2-D_2$:

$$
\begin{pmatrix} 
t-\tau -1& 0 \\
-\theta\tau^2 & t-\tau-1
\end{pmatrix}
\xrightarrow{L_1\leftrightarrow L_2} \begin{pmatrix} -\theta\tau^2 & t-\tau-1 \\ t-\tau-1 & 0\end{pmatrix}
\xrightarrow{C_1\leftrightarrow C_1(-\sigma^2\theta^{-1})}
\begin{pmatrix} 
 1&t-\tau-1 \\
 (t-\tau-1)(-\sigma^2\theta^{-1})& 0 
\end{pmatrix}$$
$$
\begin{pmatrix} 
 1&t-\tau-1 \\
 (t-\tau-1)(-\sigma^2\theta^{-1})& 0 
\end{pmatrix}\xrightarrow{\substack{C_2\rightarrow C_2-C_1(t-\tau-1)\\ L_2\rightarrow L_2-(t-\tau-1)(-\sigma^2\theta^{-1})L_1}}
\begin{pmatrix}
1 & 0 \\
0 &\lambda_2
\end{pmatrix}$$

 where $$\begin{aligned}\lambda_2&=(t-\tau-1)(-\sigma^2\theta^{-1}(t-\tau-1) \\  &=-\sigma^{2}\theta^-1 t^2+t.(\sigma^2\theta^{-1}(\tau+1)+(\tau+1)\sigma^2\theta^{-1})+-(\tau+1)\sigma^2\theta^{-1}(\tau+1).\end{aligned}$$ 
We represent the Newton polygon of $\lambda_2$ in Figure 2.
 It has only one edge of positive slope equal to $1$, hence according to Theorem \ref{3} this $t$-module is pure of weight 1. \\
 One can prove by induction that for all $n\geq 2$, $D_2^n$ has the  form:
 $D_{2}^n=\theta^n I_2+ \sum\limits_{k=1}^{n+1} A_{k,n}\tau^k$ with 
 $$A_{1,n}=\begin{pmatrix} a_{1,n} & 0\\ 0 & a_{1,n}\end{pmatrix},\  A_{k,n}=\begin{pmatrix} a_{k,n} & 0 \\ b_{k,n} &a_{k,n}\end{pmatrix}, 2\leq k\leq n-1, \ A_{n,n}=\begin{pmatrix} 1 &0\\ b_{n,n} &1\end{pmatrix}, \ A_{n+1,n}=\begin{pmatrix} 0 & 0\\ x_{n+1,n} & 0\end{pmatrix}$$ with:
 $$\left\{\begin{aligned} &a_{k_n}\in \F_q[\theta], \text{ monic of degree } q^k(n-k), \ 1\leq k \leq n-1,\\  &b_{k,n}\in \F_q[\theta], \text{ monic of degree }q^k(n-k+1)+q^{k-2}, \  2\leq k\leq n,\\ &x_{n+1,n}\in \F_q[\theta], \text{ monic of degree } q^n.\end{aligned}\right.$$ In particular, this $t$-module is not almost strictly pure.
   \end{proof}
Now we consider the general case $d>2$. We put $m:=d-2>0$ and consider the $t$-module given by the matrix
  $$D_{2+m}=\begin{pmatrix} 
    D_2 & \empty &\empty &\empty\\
    \empty &\theta+\tau &\empty &\empty \\ 
    \empty& \empty  & \ddots &\empty \\
    \empty & \empty & \empty & \theta+\tau
    \end{pmatrix}\in M_{2+m}(K\{\tau\}).$$
    
    This $t$-module is the direct sum of pure $t$-modules of weight $1$ (the $t$-module associated to $D_2$ and $d-2$ copies of the Carlitz module), so we can prove it is a pure $t$-module of weight $1$, but here we give a proof using Maurischat's algorithm. \\
    For $n\geq 1$, the leading coefficient of the matrix $D_{2+m}^n$ is the matrix
    $$\begin{pmatrix} A_{n+1,n} &\empty \\ \empty & 0_m\end{pmatrix}$$ where $0_m\in M_m(K)$ is the zero-matrix. Thus, this $t$-module is  not almost strictly pure. \\ 
    Consider $(J_0)$ the algorithm that diagonalize as previously the matrix $tI_2-D_2$. Applying $(J_0)$ and exchanging row and columns, we get the matrix:
    
 $$tI_{2+m}-D_{2+m}\longrightarrow S=\begin{pmatrix} 
    1 & \empty &\empty &\empty &\empty  \\
    \empty &t-\theta-\tau &\empty &\empty &\empty \\ 
    \empty& \empty  & \ddots &\empty &\empty \\
    \empty & \empty & \empty & t-\theta-\tau & \empty  \\
    \empty & \empty & \empty & \empty & \lambda_2  \\
    \end{pmatrix}\in M_{2+m}(K\{\tau\}[t]).$$
     
   Consider the euclidean division of $\lambda_2$ by $t-\theta-\tau$:
   $$\lambda_2=q(t-\theta-\tau)+r, \ r\neq0 \text{ and } \operatorname{deg}_t(r)=0.$$
   
   Let $$S'=\begin{pmatrix} t-\theta-\tau & \empty \\ \empty & \lambda_2\end{pmatrix}\in M_2(K\{\tau\}[t]).$$
   We apply the following operations to the matrix $S'$ (and denote by $(J_1)$ this algorithm):
   $$\begin{aligned}L_{2}&\rightarrow L_{2}-qL_{1} \\ 
   C_{2}&\rightarrow C_{2}+C_{1} \\
   L_{2}&\rightarrow r^{-1}L_{2}\\
   L_{1}&\rightarrow L_{1}-(t-\theta-\tau)L_{2}\\
   C_{2}&\rightarrow C_2 -C_{1}r^{-1}\lambda_2 \\
   L_1&\leftrightarrow L_2 \\ 
   L_2&\rightarrow -L_2.
   \end{aligned}$$
   We get the matrix:
   
   $$\begin{pmatrix} \lambda_2 & \empty \\ \empty & (t-\theta-\tau)r^{-1}\lambda_2
   \end{pmatrix}.$$
By successively applying the algorithm $(J_1)$ to the matrices $S'$ which appear from the matrix $S$, we obtain the matrix
  $$S\longrightarrow \begin{pmatrix} 
    1 & \empty &\empty &\empty &\empty  \\
    \empty &\lambda_2 &\empty &\empty &\empty \\ 
    \empty& \empty  &  (t-\theta-\tau)r^{-1}\lambda_2&\empty &\empty  \\
    \empty & \empty & \empty & \ddots& \empty  \\
    \empty & \empty & \empty & \empty & (t-\theta-\tau)r^{-1}\lambda_2  
    \end{pmatrix}.$$
    
    As $\lambda_2$ and $t-\theta-\tau$ have Newton polygons consisting of only one edge of slope 1, the Newton polygon of the last coefficient of the last matrix has only one edge of slope $1$. It follows that the Newton polygon of the last invariant factor of $D_{m+2}$ has only one edge of slope 1. Hence $D_{m+2}$ is also a $t$-module which is pure of weight 1 but not almost strictly pure for all $m\geq 0$.

    \noindent\begin{minipage}[t]{0.45\linewidth}
    \begin{figure}[H]
    \centering
     \begin{tikzpicture}
        \begin{axis}[width=0.5\linewidth]
          \addplot[black,mark=*] coordinates {(0,-3)  (2,0) };
        \end{axis}
    \end{tikzpicture}
    \caption{\centering Newton polygon of the Anderson module $M$ constructed by Maurischat when $\operatorname{char} (K)=2$ .}
    \end{figure}
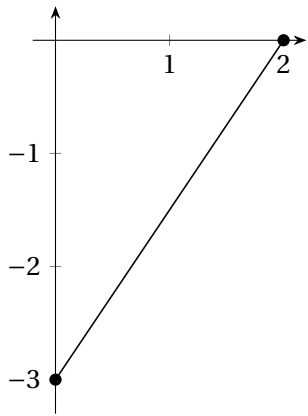
    \end{minipage}
    \hfill
    \begin{minipage}[t]{0.45\linewidth}
    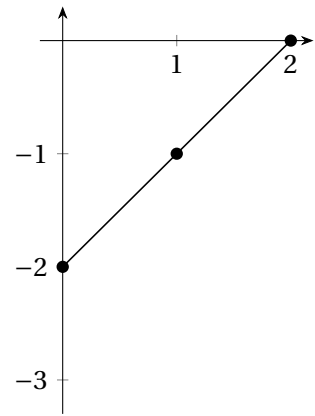
\begin{figure}[H]
    \centering
     \begin{tikzpicture}
        \begin{axis}[width=0.5\linewidth,ymin=-1]
          \addplot[black,mark=*] coordinates {(0,0) (1,1) (2,2) };
        \end{axis}
    \end{tikzpicture}
    \caption{Newton polygon of the Anderson module $D_2$ in Theorem \ref{4}.}
    \end{figure}
    \end{minipage}

\textit{\large Aknowledgements}

This result is part of my master-Thesis at University of Caen under the supervision of Floric Tavares Ribeiro and Tuan Ngo Dac that I would like to thank.
I would also like to thank Andreas Maurischat for pointing out an error in the previous version, which led to this new version.

\end{document}